\documentclass[10pt,reqno]{amsproc}
\pdfoutput=1

\usepackage{amsmath}
\usepackage{amsthm}
\usepackage{mathptmx}
\usepackage{newcent}
\usepackage{amssymb}

\usepackage{url}
\usepackage{wasysym}

\usepackage[all]{xy}

\usepackage{graphicx}

\usepackage{mathtools}

\newcommand{\blackboardbold}[1]{\ensuremath{\mathbf{#1}}}

\def\Z{\blackboardbold{Z}}
\def\Q{\blackboardbold{Q}}
\def\R{\blackboardbold{R}}

\def\H{\blackboardbold{H}}

\def\GG{\blackboardbold{G}}

\newcommand{\mathtext}[1]{\ensuremath{\mathrm{#1}}}

\newcommand{\CP}{\ensuremath{\blackboardbold{CP}}}
\newcommand{\HP}{\ensuremath{\blackboardbold{HP}}}
\newcommand{\OP}{\ensuremath{\blackboardbold{CaP}}}

\newcommand{\MO}[1]{\ensuremath{\mathtext{MO}\langle #1\rangle}}
\newcommand{\BO}[1]{\ensuremath{\mathtext{BO}\langle #1\rangle}}
\newcommand{\Ob}[1]{\ensuremath{\mathtext{O}\langle #1\rangle}}

\newcommand{\MSO}{\mathtext{MSO}}
\newcommand{\MSpin}{\mathtext{MSpin}}

\newcommand{\sgn}{\mathtext{sgn}}

\newcommand{\til}[1]{\ensuremath{\widetilde{#1}}}

\newcommand{\mrm}[1]{\mathrm{#1}}

\newcommand{\Sp}{\ensuremath{\mathrm{Sp}}}
\newcommand{\Spin}{\ensuremath{\mathrm{Spin}}}
\newcommand{\Ff}{\ensuremath{\mathrm{F}_4}}
\newcommand{\B}{\ensuremath{\mathrm{B}}}
\newcommand{\BG}{\ensuremath{\mathrm{B}G}}
\newcommand{\BH}{\ensuremath{\mathrm{B}H}}
\newcommand{\BT}{\ensuremath{\mathrm{B}T}}
\newcommand{\Bi}{\ensuremath{\mathrm{B}i}}

\renewcommand{\H}{\mathtext{H}}

\newcommand{\iso}{\ensuremath{\cong}}

\newcommand{\tensor}{\otimes}

\newcommand{\Hom}{\mathtext{Hom}}

\newcommand{\into}{\ensuremath{\hookrightarrow}}

\theoremstyle{plain}
\newtheorem*{theorem*}{Theorem}
\newtheorem{theorem}{Theorem}
\newtheorem*{proposition*}{Proposition}
\newtheorem{proposition}[theorem]{Proposition}
\newtheorem{corollary}[theorem]{Corollary}
\newtheorem*{corollary*}{Corollary}
\newtheorem{lemma}[theorem]{Lemma}
\newtheorem*{lemma*}{Lemma}

\newtheorem*{exercise*}{Exercise}

\newtheorem*{conjecture*}{Conjecture}

\newtheorem*{question*}{Question}

\theoremstyle{definition}

\newtheorem*{definition*}{Definition}

\newtheorem*{example*}{Example}
\newtheorem*{examples*}{Examples}
\newtheorem*{claim*}{Claim}

{

\newcommand{\Kummer}{\mathtext{Kum}}
\newcommand{\Proj}{\mathtext{Proj}}

\newcommand{\xto}{\xrightarrow}
\newcommand{\xfrom}{\xleftarrow}
\newcommand{\xinto}{\xhookrightarrow} % from the mathtools package

\allowdisplaybreaks[1]

\title{The Cayley Plane and the Witten Genus}

\author{Carl McTague}
\email{c.mctague@dpmms.cam.ac.uk}
\address{DPMMS, Wilberforce Road, Cambridge CB3 0WB, England}

\begin{document}

\begin{abstract}
  This paper defines a new genus, the \emph{Cayley plane genus}. By
  definition it is the universal multiplicative genus for oriented
  Cayley plane bundles. The main result
  (Theorem~\ref{thm:cayley-plane-genus}) is that it factors (tensor
  $\Q$) through the product of the Ochanine elliptic genus and the
  Witten genus---revealing a synergy between these two genera---and
  that its image is the homogeneous coordinate ring:
  \begin{align*}
    \Q[\Kummer,\HP^2,\HP^3,\OP^2] \Big/ \big(\OP^2\big) \cdot
    \big(\HP^3, \OP^2 - (\HP^2)^2 \big)
  \end{align*}
  of the union of the curve of Ochanine elliptic genera and the
  surface of Witten genera meeting with multiplicity~2 at the point
  $\OP^2=\HP^3=\HP^2=0$ corresponding to the $\hat{A}$-genus. This all
  remains true if the word ``oriented'' is replaced with the word
  ``spin'' (Theorem~\ref{thm:mspin-case}). This paper also
  characterizes the Witten genus (tensor $\Q$) as the universal genus
  vanishing on total spaces of Cayley plane bundles
  (Theorem~\ref{thm:new-char-witten-genus}, a result proved
  independently by Dessai in \cite{dessai-2009}.)
\end{abstract}

\maketitle

\section{Introduction}

This paper is inspired by two theorems. The first was proved in the
1950's.

\begin{theorem*}[Chern-Hirzebruch-Serre
  \cite{chern-hirzebruch-serre-1957}, Borel-Hirzebruch
  \cite{borel-hirzebruch-1959}] The universal multiplicative genus for
oriented manifolds is the signature:
  \begin{align*}
    \MSO_* \tensor \Q \to \MSO_* / (E-F \cdot B) \tensor \Q \iso
    \Q[\sigma]
  \end{align*}
  where $\deg(\sigma)=4$.
\end{theorem*}
Here and throughout this paper $(E-F\cdot B)$ denotes the $\Q$-vector
space span\-ned by differences $E-F\cdot B$ where $F \to E \to B$ ranges
over all fiber bundles with compact connected structure group. This
vector space is in fact an ideal. The ambient bordism ring and hence
the nature of the manifolds $F,E,B$ will vary.

This theorem encapsulates several results. First of all Hirzebruch
\cite{hirzebruch-1956} established a correspondence between genera
$\MSO_* \to \Q$ and formal power series
$Q(z)=1+a_2\,z^2+a_4\,z^4+\cdots \in \Q[[z^2]]$. He then showed, in
his celebrated Signature Theorem, that the signature is the genus
corresponding to the power series:
\begin{align*}
  Q(z)=z/\tanh(z)=1+\tfrac13\,z^2-\tfrac1{45}\,z^4+\cdots
\end{align*}
Next Chern-Hirzebruch-Serre \cite{chern-hirzebruch-serre-1957} proved
that the signature is multiplicative, that is
$\sigma(E)=\sigma(B)\,\sigma(F)$, for oriented fiber bundles with
connected structure group (or more generally with $\pi_1(B)$ acting
trivially on $\H^*(F,\R)$).  Finally Borel-Hirzebruch \cite[Theorem
28.4]{borel-hirzebruch-1959} showed that the signature is the only
multiplicative genus for oriented fiber bundles. Note that the Euler
characteristic is multiplicative for oriented fiber bundles but is not
an oriented bordism invariant (it is a complex bordism invariant).
Totaro \cite{totaro-2007} articulated the theorem as written above.

\medskip

The second theorem which inspired this paper was proved in the 1980's.

\begin{theorem*}[Ochanine \cite{ochanine-1987}, Bott-Taubes
  \cite{bott-taubes-1989}]
  The universal multiplicative genus for spin manifolds is the
  Ochanine elliptic genus:
  \begin{align*}
    \phi_{ell} : \MSpin_* \tensor \Q \to \MSpin_* / (E-F\cdot B)
    \tensor \Q \iso \Q[\delta,\epsilon]
  \end{align*}
  which maps onto the ring of modular forms on the congruence subgroup
  $\Gamma_0(2)$. In particular $\deg(\delta)=2, \deg(\epsilon)=4$.
\end{theorem*}

More concretely there is a family of $\Q$-valued multiplicative genera
for spin manifolds and the members of this family correspond to stable
elliptic curves with a marked point of order~2, the points of the
weighted projective moduli space $\Proj\;\Q[\delta,\epsilon]$. In
particular the logarithms $g(y)=(y/Q(y))^{-1}$ of these genera are
elliptic integrals:
\begin{align*}
  \int_0^y \frac{dt}{\sqrt{1-2\delta t^2+\epsilon t^4}}
\end{align*}
The special cases $[\delta,\epsilon]=[1,1]$ and $[-\tfrac18,0]$ are
the signature and $\hat{A}$~genus respectively. The discriminant
$\Delta=64\epsilon^2(\delta^2-\epsilon)$ vanishes in these cases so
they correspond to singular elliptic curves.

A more elegant approach is to gather the entire family into a single
characteristic power series whose coefficients are modular forms:
\begin{align*}
  Q_{ell}(z) = \exp \left( \sum_{k=1}^\infty \frac{2}{(2k)!}
    \til{\GG}_{2k}\,z^{2k} \right)
\end{align*}
where $\til{\GG}_k$ denotes the Eisenstein series:
\begin{align*}
  \til{\GG}_k &= -\frac{1}{2k} \mathrm{B}_k + \sum_{n \ge 1} \Big(
  \sum_{d|n} (-1)^{n/d} d^{k-1} \Big) q^n
\end{align*}
of weight $k$ on the congruence subgroup $\Gamma_0(2)$ (see
\cite{zagier-86}).

% Note that $\delta=3\til{\GG}_2$ and
% $\epsilon=\tfrac16(12\til{\GG}_2^2-5\til{\GG_4})$ that
% $\til{\GG}_{2k}$ for $k>2$ can be expressed as a polynomial in
% $\Q[\til{\GG}_2,\til{\GG}_4]$. For example:
% \begin{align*}
%   \til{\GG}_6 &= \tfrac{120}7(4\til{\GG}_2^3-\til{\GG}_2\til{\GG}_4) &
%   \til{\GG}_8&=
%   -\tfrac{20}3(144\til{\GG}_2^4-120\til{\GG}_2^2+7\til{\GG}_4^2)
% \end{align*}

\smallskip

Inspired by this characteristic power series, Witten
\cite{witten-1986} introduced the characteristic power series:
\begin{align*}
  Q_W(z) = \exp \left( \sum_{k=1}^\infty \frac{2}{(2k)!}
    \GG_{2k}\,z^{2k} \right)
\end{align*}
where $\GG_k$ denotes the Eisenstein series:
\begin{align*}
  \GG_k &= -\frac{1}{2k} \mathrm{B}_k + \sum_{n \ge 1} \Big(
  \sum_{d|n} d^{k-1} \Big) q^n = \frac{(k-1)!}{(2\pi i)^k}
  \mathrm{G}_k = \frac{\zeta(k) (k-1)!}{(2\pi i)^k} \mathrm{E}_k
\end{align*}
of weight $k$ on the full modular group $\mathrm{PSL}(2,\Z)$.
% Note that $\til{\GG}_{2k}$ for $k>3$ can be expressed as a polynomial
% in $\Q[\GG_2,\GG_4]$. For example $\GG_8&=120\GG_4^2$.
This defines the \emph{Witten genus:}
\begin{align*}
  \phi_W : \MSO_* \tensor \Q \to \Q[\GG_2,\GG_4,\GG_6]
\end{align*}
which maps onto the ring of quasi-modular forms ($\GG_2$ is not
modular).

% Note that the weighted projective space
% $\Proj\;\Q[\GG_2,\GG_4,\GG_6]$ fibers over the moduli space
% $\Proj\;\Q[\GG_4,\GG_6]$ of stable elliptic curves.

\bigskip

Now the bordism rings $\MSO_*$ and $\MSpin_*$ are the second and third
terms in an infinite sequence:
\begin{align*}
  \mathrm{MO}_* \quad \MO2_* \quad \MO4_* \quad \MO8_* \quad \MO9_*
  \quad \cdots
\end{align*}
Here $\MO{n}_*$ denotes the bordism ring of $\Ob{n}$ manifolds. An
$\Ob{n}$ manifold is a smooth manifold $M$ equipped with a lift of its
stable tangent bundle's classifying map to the $(n-1)$-connected cover
$\BO{n}$ of the classifying space $\mathrm{BO}$:
\begin{align*}
  \xymatrix{ & \BO{n} \ar[d] \\
    M \ar[r] \ar@{-->}[ur] & \mathrm{BO} }
\end{align*}
The integers appearing in the sequence come from Bott periodicity:
\begin{align*}
  \pi_i(\mathrm{BO}) =
  \begin{cases}
    \Z/2 & \text{for $i=1,2$ mod $8$} \\
    \Z & \text{for $i=4,8$ mod $8$} \\
    0 & \text{otherwise}
  \end{cases}
\end{align*}
The fourth term in the sequence $\MO8_*$ is sometimes denoted
$\mathrm{MString}_*$. An $\Ob{8}$ manifold can be characterized as a
spin manifold whose characteristic class $\tfrac12p_1$, the pullback
of the generator of $\H^4(\BO8,\Z)$, equals zero.

\smallskip

In light of this sequence of bordism rings the two theorems above
suggest the following question.

\begin{question*}
  What is the universal multiplicative genus for $\Ob{8}$ manifolds?
  \begin{align*}
    \MO8_* \tensor \Q \to \MO8_*/(E-F\cdot B) \tensor \Q
  \end{align*}
\end{question*}

This paper gives a first approximation to the answer. The idea came
while reading Hirzebruch's textbook \cite{hirzebruch-92}. In \S4.6 he
shows that although the natural habitat of the elliptic genus is
$\MSpin_*$, it can already be observed in $\MSO_*$. The result is
originally due to Ochanine \cite{ochanine-1987}.

\begin{theorem*}[Ochanine]
  \begin{align*}
    \MSO_* / (E-\CP^2 \cdot B) \tensor \Q &\iso \Q[\sigma] \\
    \MSO_* / (E-\CP^3 \cdot B) \tensor \Q &\iso \Q[\delta,\epsilon]
  \end{align*}
\end{theorem*}

The point is that $\CP^2$ and $\CP^3$ both have lots of automorphisms
and therefore are fibers of lots of bundles. But $\CP^3$ is spin
whereas $\CP^2$ is not.

This made me wonder what would happen if I replaced $\CP^3$ with some
$\Ob{8}$ manifold having lots of automorphisms. The Cayley plane
$\OP^2=\Ff/\Spin(9)$ is such a manifold. Sometimes denoted
$\mathbf{OP}^2$ it is in a certain sense a projective plane over the
octonions (see \cite[\S12.2]{conway-smith-2003}). So I set out to
compute the quotient:
\begin{align*}
  \MSO_* /(E-\OP^2\cdot B) \tensor \Q
\end{align*}
I began by doing explicit power series calculations in the spirit of
Hirzebruch's textbook \cite{hirzebruch-92} to determine all strictly
multiplicative genera for Cayley plane bundles. My calculations
strongly suggested that there were two families of strictly
multiplicative genera, and I recognized them as the elliptic genus and
the Witten genus.

I was not surprised to find the elliptic genus and the Witten genus.
As the second theorem above states, the elliptic genus is known to be
multiplicative not only for Cayley plane bundles but for \emph{any}
oriented fiber bundle with fiber a spin manifold and compact connected
structure group. This is a consequence of its rigidity
\cite{bott-taubes-1989}.  The Witten genus is also rigid but it is
multiplicative in an even starker sense: the Witten genus of any
Cayley plane bundle, and more generally any bundle with fiber a
homogeneous space which is $\Ob8$, is zero (see \cite[Theorem
3.1]{stolz96}).

I \emph{was} surprised, however, to find \emph{only} the elliptic
genus and the Witten genus. This clue led me to the following two
theorems.

\begin{theorem}
  \label{thm:new-char-witten-genus}
  \label{THM:NEW-CHAR-WITTEN-GENUS}
  The Witten genus is the universal genus vanishing on Cayley plane
  bundles. In other words the Witten genus is the quotient map:
  \begin{align*}
    \phi_W : \MSO_* \tensor \Q \to \MSO_* / (E) \tensor \Q \iso
    \Q[\GG_2,\GG_4,\GG_6]
  \end{align*}
  where $(E) \subset \MSO_* \tensor \Q$ denotes the $\Q$-vector
  space spanned by total spaces of Cayley plane bundles $\OP^2 \to E
  \to B$ with compact connected structure group. (This vector space is
  an ideal.)
\end{theorem}

\begin{theorem}
  \label{thm:cayley-plane-genus}
  \label{THM:CAYLEY-PLANE-GENUS}
  The universal multiplicative genus for Cayley plane bundles is the
  product $\phi_{ell} \times \phi_W$ of the Ochanine elliptic genus
  and the Witten genus. More precisely, the quotient $\MSO_* /(E-\OP^2
  \cdot B) \tensor \Q$ injects into $\Q[\delta,\epsilon] \times
  \Q[\GG_2,\GG_4,\GG_6]$ and the composition:
  \begin{align*}
      \MSO_* \tensor \Q \to
      \MSO_* / (E-\OP^2 \cdot B) \tensor \Q \into
      \Q[\delta,\epsilon] \times \Q[\GG_2,\GG_4,\GG_6]
  \end{align*}
  is $\phi_{ell} \times \phi_W$. Its image can be described
  geometrically as the weighted homogeneous coordinate ring:
  \begin{align*}
    \Q[\Kummer,\HP^2,\HP^3,\OP^2] \Big/ \left(\OP^2\right) \cdot
    \left(\HP^3, \OP^2 - (\HP^2)^2 \right)
  \end{align*}
  of the union of the weighted projective spaces:
  \begin{align*}
    \Proj\;\Q[\delta,\epsilon] & \xfrom[\iso]{\phi_{ell}}
    \Proj\;\Q[\Kummer,\HP^2,\HP^3,\OP^2]/(\HP^3,\OP^2-(\HP^2)^2) \\
    \Proj\;\Q[\GG_2,\GG_4,\GG_6] & \xfrom[\iso]{\phi_W}
    \Proj\;\Q[\Kummer,\HP^2,\HP^3,\OP^2]/(\OP^2)
  \end{align*}
  The first is the curve of elliptic genera (the moduli space of
  stable elliptic curves with a marked 2-division point). The second
  is the surface of Witten genera (related to the moduli space of
  stable elliptic curves). They meet with multiplicity~2 at the point
  $\OP^2=\HP^3=\HP^2=0$ corresponding to the $\hat{A}$~genus.
\end{theorem}

I should emphasize that $\phi_{ell} \times \phi_W$ is \emph{not}
surjective. Indeed an essential point is that:
\begin{align*}
  \phi_{ell} \times \phi_W (\OP^2) &= (\epsilon^2,0)
\end{align*}
but that $(\epsilon,0)$ is not in the image of $\phi_{ell} \times
\phi_W$. For instance:
\begin{align*}
  \phi_{ell} \times \phi_W(\HP^2) = (\epsilon,2\GG_2^2-\tfrac56\GG_4)
\end{align*}
Thus there is a synergy between the elliptic genus and the Witten
genus: individually they cannot recognize $\OP^2$ as an indecomposable
but together they can.

Note that the values:
\begin{align*}
  \phi_W(\OP^2)&=0 & \phi_{ell}(\HP^3)&=0 &
  \phi_{ell}(\OP^2)&=\epsilon^2=\phi_{ell}(\HP^2)^2
\end{align*}
together with:
\begin{align*}
  \phi_{ell} \times \phi_W(\Kummer)&=(16 \delta,48 \GG_2) &
  \phi_{ell} \times \phi_W (\HP^3)
  &=(0,-\tfrac49\GG_2^3+\tfrac19\GG_2\GG_4+\tfrac7{1080}\GG_6)
\end{align*}
account for the isomorphisms of weighted projective spaces asserted in
the theorem (compare Proposition~\ref{prop:phiEllxphiWValues}).

\bigskip

There is further evidence that Theorems
\ref{thm:new-char-witten-genus} \& \ref{thm:cayley-plane-genus} are a
good approximation of the answer to the \textsc{Question} above.

\begin{theorem}
  \label{thm:mspin-case}
  \label{THM:MSPIN-CASE}
  Theorems \ref{thm:new-char-witten-genus} \&
  \ref{thm:cayley-plane-genus} remain true if $\MSO_*$ is replaced
  with $\MSpin_*$.
\end{theorem}

Note that the \textsc{Question} would be answered if $\MSpin_*$ could
be replaced with $\MO8_*$. (The description of the image in
Theorem~\ref{thm:cayley-plane-genus} would need to be modified
though.)

Note also that Dessai proved Theorem~\ref{thm:new-char-witten-genus}
independently in \cite{dessai-2009}. In fact he showed that it remains
true if $\MSO_*$ is replaced with $\MO8_*$ (in which case the
generator $\GG_2$ should be erased). Note that Dessai asked
(Problem~4.2 of his paper) for a geometric description of the
universal multiplicative genus for $\OP^2$ bundles.
Theorem~\ref{thm:cayley-plane-genus} answers that question as
stated. However, I expect a richer answer to come from replacing
$\MSpin$ with $\MO8$ in Theorem~\ref{thm:cayley-plane-genus}.

\medskip

I conclude the introduction by speculating about how these results
might be relevant to homotopy theory. Kreck-Stolz
\cite{kreck-stolz-93} computed:
\begin{align*}
  &\MSpin_* / (E) \iso \mrm{KO}_*(\mrm{pt})
  &\MSpin_* / (E-\HP^2 \cdot B) \tensor \Z[\tfrac12] \iso
  \Z[\tfrac12][\delta,\epsilon]
\end{align*}
where in both cases $\HP^2 \to E \to B$ ranges over all bundles with
compact connected structure group. They used these calculations to
give alternate constructions of $\mrm{KO}$-theory and elliptic
homology (and in so doing defined elliptic cohomology with $\Z$ rather
than $\Z[\tfrac12]$ coefficients, which was novel). They suggested
(see \cite[p.~235]{kreck-stolz-93}) that replacing $\MSpin_*$ and
$\HP^2$ with $\MO8_*$ and $\OP^2$ in their constructions might result
in homology theories as well. (Sati's recent paper \cite{sati-2009}
explores the relevance of such a theory to string theory.)
Theorems~\ref{thm:new-char-witten-genus} \&
\ref{thm:cayley-plane-genus} suggest that the first might be closely
related to topological modular forms \cite{hopkins-2002} while the
second might be some sort of hybrid of elliptic homology and
topological modular forms.

\section{Cayley plane bundles}

Before we can prove Theorems \ref{thm:new-char-witten-genus} \&
\ref{thm:cayley-plane-genus}, we need to discuss Cayley plane bundles
in general. The Cayley plane is the homogeneous space
$\OP^2=\Ff/\Spin(9)$. Much of what follows applies to any bundle with
fiber a homogeneous space $G/H$ though so we begin in that generality
and later specialize to the case $G/H=\Ff/\Spin(9)$.

Throughout this section let $G$ be a compact connected Lie group, let
$i_{H,G} : H \into G$ be a maximal rank subgroup, and let $i_{T,H} : T
\to H$ and $i_{T,G} : T \to G$ be the inclusions of a common maximal
torus.

Every $G/H$ bundle (with structure group $G$) pulls back from the
universal $G/H$ bundle $G/H \to \BH \to \BG$. That is, every $G/H$
bundle fits into a pullback diagram:
\begin{align*}
  \xymatrix{ E_f \ar[r]^-g \ar[d]_{\pi_f}
    & \BH \ar[d]^{\Bi_{H,G}} \\
    Z \ar[r]^-f & \BG }
\end{align*}
where $f$ is unique up to homotopy and $g$ is canonically determined
by $f$.

Let $\eta$ denote the relative tangent bundle of $\BH \to \BG$.
Then the relative tangent bundle of $E_f \to Z$ is the pullback
$g^*(\eta)$ and there is an exact sequence:
\begin{align*}
%  \label{eqn:exct-sqnc}
  0 \to g^*(\eta) \to TE_f \to \pi_f^* TZ \to 0
\end{align*}
This implies for instance that $p_1(TE_f)=\pi_f^*p_1(TZ) +
g^*p_1(\eta)$.

The characteristic classes of $\eta$, or rather their pullbacks to
$\H^*(\BT,\Z)$, can be computed using the beautiful methods of
\cite{borel-hirzebruch-1958} (see especially Theorem~10.7).  For
instance the pullbacks of the first Pontrjagin class $p_1(\eta)$ and
more generally the Pontrjagin class $s_I(p)(\eta)$ can be computed
using the formulas:
\begin{align*}
  \Bi_{T,H}^* p_1(\eta) = \sum r_i^2 &&
  \Bi_{T,H}^* s_I(p)(\eta) = s_I(r_1^2,\dots,r_m^2)
\end{align*}
where $(\pm r_1,\dots,\pm r_m)$ are the roots of $G$ complementary to
those of $H$ regarded as elements of $\H^*(\BT,\Z)$.

% Here is how to regard a root of $G$ as an element of $\H^2(\BT)$.
% Identify $\pi_1(T)$ with the integral lattice $\Gamma=\ker(\exp : LT
% \to T)$ and note that:
% \begin{align*}
%   \Hom(\Gamma,\Z) \iso \Hom(\pi_1(T),\Z) \iso \H^1(T,\Z) \iso
%   \H^2(\BT,\Z)
% \end{align*}
% The last isomorphism comes from transgression in the Serre spectral
% sequence for $T \to ET \to BT$.  An (infinitesimal real) root of $G$
% is a linear form $r : LT \to \R$ and a basic result says that
% $r(\Gamma) \subset \Z$. In fact a theorem says that $\pi_1(T) \to
% \pi_1(G)$ is surjective with kernel generated by the coroots of $G$.

\bigskip

Borel-Hirzebruch's Lie-theoretic description
\cite{borel-hirzebruch-1958,borel-hirzebruch-1959} of the pushforward:
\begin{align*}
  \Bi_{H,G*} : \H^*(\BH,\Z) \to \H^*(\BG,\Z)
\end{align*}
is essential to proving Theorems~\ref{thm:new-char-witten-genus} \&
\ref{thm:cayley-plane-genus}.  In order to state their result we need
to introduce some notation.

Associated to $G$ is a generalized Euler class $\til{e}(G/T) \in
\H^*(\BT,\Z)$. It makes sense to call it that because it restricts to
the Euler class of the fiber $G/T$ of the bundle $\BT \to \BG$. Up to
sign $\til{e}(G/T)$ is the product of a set of positive roots of $G$,
regarded as elements of $\H^*(\BT,\Z)$.  More precisely it is the
product of the roots of an invariant almost complex structure on
$G/T$. (See \cite[\S12.3, \S13.4]{borel-hirzebruch-1958} for more
details.)  Note that $G/T$ always admits a complex structure and that
although the individual roots associated to an almost complex
structure depend on the almost complex structure, their product
$\til{e}(G/T)$ does not.

\begin{theorem}[Borel-Hirzebruch, Theorem~20.3 of
  \cite{borel-hirzebruch-1959}]
  \label{thm:borel-hirzebruch}
  If $t \in \H^*(\BT,\Z)$ then:
  \begin{align*}
    \Bi_{T,G}^* \Bi_{T,G*}(t) = \frac{1}{\til{e}(G/T)} \sum_{w \in W(G)}
    \sgn(w) \; w(t)
  \end{align*}
\end{theorem}

\begin{corollary}
  \label{BHToBG}
  \label{cosetpushforward}
  If $h \in \H^*(\BH,\Z)$ then:
  \begin{align*}
    \Bi_{T,G}^* \Bi_{H,G*}(h)
%
%    = \frac{1}{|W(H)|} \sum_{w \in W(G)} w\left(
%      \frac{\til{e}(H/T)}{\til{e}(G/T)} \Bi_{T,H}^*(h) \right)
%
    = \sum_{[w] \in W(G)/W(H)} w\left(
      \frac{\til{e}(H/T)}{\til{e}(G/T)} \Bi_{T,H}^*(h) \right)
  \end{align*}
  where the sum runs over the cosets of $W(H)$ in $W(G)$.
\end{corollary}
\begin{proof}
%  This is a nice trick from Hirzebruch's textbook
%  \cite[p.~51]{hirzebruch-92}.
  Since $\Bi_{T,H*} \til{e}(H/T)=\chi(H/T)=|W(H)| \in \H^0(\BH,\Z)$,
  write:
  \begin{align*}
    \Bi_{T,G}^* \Bi_{H,G*}(h) &= \Bi_{T,G}^* \Bi_{H,G*} \left(
      \frac{\Bi_{T,H*}(\til{e}(H/T))}{|W(H)|} \cdot h \right) \\
    \intertext{Apply the projection formula:}
    &= \frac{1}{|W(H)|} \Bi_{T,G}^* \Bi_{H,G*} \; \Bi_{T,H*} \left(
      \til{e}(H/T) \cdot \Bi_{T,H}^*(h) \right) \\
    &= \frac{1}{|W(H)|} \Bi_{T,G}^* \Bi_{T,G*} \left(\til{e}(H/T) \cdot
      \Bi_{T,H}^*(h) \right) \\
    \intertext{Apply Theorem~\ref{thm:borel-hirzebruch}:}
    &= \frac{1}{|W(H)|} \cdot \frac{1}{\til{e}(G/T)} \sum_{w \in W(G)}
    \sgn(w) \; w(\til{e}(H/T) \cdot \Bi_{T,H}^*(h)) \\
    \intertext{Since $w(\til{e}(G/T))=\sgn(w)\til{e}(G/T)$:}
    &= \frac{1}{|W(H)|} \sum_{w \in W(G)} w\left(
      \frac{\til{e}(H/T)}{\til{e}(G/T)} \Bi_{T,H}^*(h) \right)
    \intertext{Since $W(G)$ acts on $\H^*(\BT,\Z)$ by ring
      homomorphisms, since if $w \in W(H)$ then
      $w(\til{e}(H/T))=\sgn(w)\til{e}(H/T)$ and
      $w(\til{e}(G/T))=\sgn(w)\til{e}(G/T)$, and since $\Bi_{T,H}^*$
      maps to the $W(H)$-invariant subring of $\H^*(\BT,\Z)$, this sum
      can be written over the cosets of $W(H)$ in $W(G)$:}
    &= \sum_{[w] \in W(G)/W(H)} w\left(
      \frac{\til{e}(H/T)}{\til{e}(G/T)} \Bi_{T,H}^*(h) \right)
    \tag*{\qedhere}
  \end{align*}
\end{proof}

% \begin{corollary}
%   \label{cosetpushforward}
%   If an element $h \in \H^*(\BH,\Z)$ satisfies
%   $w(\Bi_{T,H}^*(h))=\Bi_{T,H}^*(h)$ for all $w \in W(H)$ then the sum
%   of Corollary~\ref{BHToBG} can be written over the cosets of $W(H)$
%   in $W(G)$:
%   \begin{align*}
%     \Bi_{T,G}^* \Bi_{H,G*}(h)
%     = \sum_{[w] \in W(G)/W(H)} w\left(
%       \frac{\til{e}(H/T)}{\til{e}(G/T)} \Bi_{T,H}^*(h) \right)
%   \end{align*}
%   This condition is satisfied if $\Bi_{T,H}^*(h)$ is a symmetric
%   polynomial (with say integer coefficients) in the squares of roots
%   of $G$ complementary to those of $H$.
% \end{corollary}

% \begin{proof}
%   The first statement follows immediately from Corollary~\ref{BHToBG}.
%   To prove the second statement consider any element $w \in W(H)$.
%   Note that $w(\til{e}(H/T))= \sgn(w) \til{e}(H/T)$ and
%   $w(\til{e}(G/T))=\sgn(w)\til{e}(G/T)$.  Also note that $w$ permutes
%   the roots of $H$ and therefore permutes the roots of $G$
%   complementary to those of $H$. So if $\Bi_{T,H}^*(h)$ is a symmetric
%   polynomial in the squares of the complementary roots then
%   $w(\Bi_{T,H}^*(h))=\Bi_{T,H}^*(h)$. The result follows since $W(G)$
%   acts on $\H^*(\BT,\Z)$ by ring homomorphisms.
% \end{proof}

\bigskip

Now we specialize to Cayley plane bundles. Let $\Ff$ denote the
1-connected compact Lie group of type $\Ff$. The extended Dynkin
diagram of $\Ff$ is:
\begin{align*}
  &\xymatrix@R1pt@!{ \bullet \ar@{-}[r] & \circ \ar@{-}[r] 
    & \circ \ar@{=}[r] |*=0{>} & \circ \ar@{-}[r] & \circ \\
    -\til{a} & a_1 & a_2 & a_3 & a_4 }
\intertext{The corresponding simple roots can be taken to be:
  \begin{align*}
    a_1 = e_2-e_3 \quad\quad\quad
    a_2=e_3-e_4 \quad\quad\quad
    a_3=e_4 \quad\quad\quad
    a_4=\tfrac12(e_1-e_2-e_3-e_4)
  \end{align*}
  Since the coefficient of $a_4$ in the maximal root
  $\til{a}=2a_1+3a_2+4a_3+2a_4=e_1+e_2$ is prime, a theorem of Borel
  \& de Siebenthal \cite{borel-de-siebenthal-1949} implies that
  erasing $a_4$ from the extended Dynkin diagram gives the Dynkin
  diagram of a subgroup:}
&\xymatrix@R1pt{ \circ \ar@{-}[r] & \circ \ar@{-}[r] 
    & \circ \ar@{=}[r] |*=0{>} & \circ \\
    -\til{a} & a_1 & a_2 & a_3 }
\end{align*}
Since $\Ff$ is 1-connected this subgroup is $\Spin(9)$, the
1-connected double cover of $\mathrm{SO}(9)$.  The Cayley plane is the
homogeneous space $\OP^2 = \Ff/\Spin(9)$.

\medskip

In terms of the standard basis $e_1,\dots,e_4$, the roots of $\Spin(9)$
are:
\begin{align*}
  \begin{cases}
    \pm e_i & 1 \le i \le 4 \\
    \pm e_i \pm e_j & 1 \le i < j \le 4
  \end{cases}
\end{align*}
The roots of $\Ff$ are those of $\Spin(9)$ together with the
complementary roots:
\begin{align*}
  \tfrac12 (\pm e_1 \pm e_2 \pm e_3 \pm e_4)
\end{align*}

The following positive roots define an almost complex structure on
$\Spin(9)/T$:
\begin{align*}
  \begin{cases}
    e_i & 1 \le i \le 4 \\
    e_i \pm e_j & 1 \le i < j \le 4
  \end{cases}
\end{align*}
These positive roots together with the following complementary
positive roots define an almost complex structure on $\Ff/T$:
\begin{align*}
  r_i := \tfrac12(e_1 \pm e_2 \pm e_3 \pm e_4)
  \quad \text{for $1 \le i \le 8$}
\end{align*}

In order to identify these roots with elements of $\H^2(\BT,\Z) \iso
\Hom(\Gamma,\Z)$ note that in general a Lie group's lattice of
integral forms is sandwiched somewhere between its root and weight
lattices:
\begin{align*}
  R \subset \Hom(\Gamma,\Z) \subset W \subset LT^*
\end{align*}
But in the case of $\Ff$ all three lattices coincide (because the
Cartan matrix of $\Ff$ has determinant 1).

% So $\H^*(\BT,\Z)$ can be identified with the polynomial ring
% $\Z[y_1,y_2,y_3,y_4]$ where:
% \begin{align*}
%   y_1=e_1, y_2=e_2, y_3=e_3, y_4=\tfrac12(e_1+e_2+e_3+e_4)
% \end{align*}
% are the fundamental weights of $\Ff$. Note then that
% $e_4=2y_4-y_1-y_2-y_3$.

\medskip

Finally note that if $s_i$ denotes reflection across the hyperplane
orthogonal to the simple root $a_i$ then the 3 cosets of $W(\Spin(9))$
in $W(\Ff)$ can be represented by the reflections $\{1,s_4,s_4 s_3
s_4\}$ which act on $e_1,\dots,e_4$ according to the matrices:
\begin{align*}
  \left\{ 
    \begin{pmatrix}
      1 & 0 & 0 & 0 \\
      0 & 1 & 0 & 0 \\
      0 & 0 & 1 & 0 \\
      0 & 0 & 0 & 1
    \end{pmatrix}, \;
    \frac12
    \begin{pmatrix}
      1 & 1 & 1 & 1 \\
      1 & 1 & -1 & -1 \\
      1 & -1 & 1 & -1 \\
      1 & -1 & -1 & 1
    \end{pmatrix}, \;
    \frac12
    \begin{pmatrix}
      1 & 1 & 1 & - 1 \\
      1 & 1 & - 1 & 1 \\
      1 & - 1 & 1 & 1 \\
      -1 & 1 & 1 & 1
    \end{pmatrix}
  \right\}
\end{align*}
In particular these reflections act on the set of positive
complementary roots $r_i$ by:
\begin{align*}
  \{r_i\} & = 
    \{\tfrac12(e_1 \pm e_2 \pm e_3 \pm e_4)\} \\
    s_4(\{r_i\}) &= \{ e_1,e_2,e_3,e_4, \tfrac12(e_1+e_2+e_3-e_4),
    \tfrac12(e_1+e_2-e_3+e_4),\\&\quad\quad \tfrac12(e_1-e_2+e_3+e_4),
    \tfrac12(-e_1+e_2+e_3+e_4) \} \\
    s_4s_3s_4(\{r_i\}) &= \{ e_1,e_2,e_3,e_4,
    \tfrac12(e_1+e_2+e_3+e_4), \tfrac12(e_1+e_2-e_3-e_4),
    \\&\quad\quad \tfrac12(e_1-e_2+e_3-e_4),
    \tfrac12(-e_1+e_2+e_3-e_4) \}
\end{align*}

\begin{corollary}
  \label{spin9f4pushforward}
  \begin{align*}
    \Bi_{T,\Ff}^* \Bi_{\Spin(9),\Ff*} s_I(p)(\eta) &=
    \frac{s_I(r_1^2,\dots,r_8^2)}{\prod_i r_i} + s_4 \left(
      \frac{s_I(r_1^2,\dots,r_8^2)}{\prod_i r_i} \right) + s_4 s_3 s_4
    \left( \frac{s_I(r_1^2,\dots,r_8^2)}{\prod_i r_i} \right)
  \end{align*}
  where the complementary roots $r_i = \tfrac12(e_1 \pm e_2 \pm e_3
  \pm e_4)$ are regarded as elements of $\H^2(\BT,\Z)$ and
  $s_4,s_4s_3s_4$ act on them as described above.
\end{corollary}

\section{Proof of Theorem~\ref{THM:NEW-CHAR-WITTEN-GENUS}}

Theorem~\ref{thm:new-char-witten-genus} is a consequence of
Proposition~\ref{prop:kill-higher-gens} together with the calculation
(Proposition~\ref{prop:phiEllxphiWValues} in the next section) of the
Witten genus of $\Kummer,\HP^2,\HP^3$.

\begin{proposition}
  \label{prop:kill-higher-gens}
  If $n \ge 4$ then there is a Cayley plane bundle $\OP^2 \to E_n \to
  \HP^{n-4}$ with $s_n(p)[E_n] \ne 0$.
\end{proposition}

The proof relies on the following lemma (Lemma~16.2 of
\cite{milnor-stasheff74}).
\begin{lemma}[Thom]
  \label{lemma:thom}
  If $0 \to V_1 \to W \to V_2 \to 0$ is an exact sequence of vector
  bundles then:
  \begin{align*}
    s_I(p)(W) = \sum_{JK=I} s_J(p)(V_1) \; s_K(p)(V_2)
  \end{align*}
  where the sum ranges over all partitions $J$ and $K$ with
  juxtaposition $JK$ equal to $I$.
\end{lemma}

\begin{proof}[Proof of Proposition~\ref{prop:kill-higher-gens}]
  Recall that the extended Dynkin diagram of $\Ff$ is:
  \begin{align*}
    &\xymatrix@R1pt@!{ \bullet \ar@{-}[r] & \circ \ar@{-}[r] 
      & \circ \ar@{=}[r] |*=0{>} & \circ \ar@{-}[r] & \circ \\
      -\til{a} & a_1 & a_2 & a_3 & a_4 }
    \intertext{Since the coefficient of $a_1$ in the maximal root
      $\til{a}=2a_1+3a_2+4a_3+2a_4$ is prime, a theorem of
      Borel-Siebenthal \cite{borel-de-siebenthal-1949} implies that
      erasing $a_1$ from the extended Dynkin diagram gives the Dynkin
      diagram of a subgroup. This subgroup's (half) extended Dynkin
      diagram is:}
    &\xymatrix@R1pt@!{ \circ && \circ \ar@{=}[r] |*=0{>} & \circ
      \ar@{-}[r] & \circ \ar@{=}[r]|*=0{<} & \bullet \\
      -\til{a} && a_2 & a_3 & a_4 & -\til{b}}
    \intertext{Since the coefficient of $a_3$ in the maximal root
      $\til{b}=2a_4+2a_3+a_2$ is prime, the same theorem implies that
      $\Ff$ has a subgroup with Dynkin diagram:}
    &\xymatrix@R1pt@!{ \circ && \circ && \circ \ar@{=}[r]|*=0{<} & \circ \\
      -\til{a} && a_2 && a_4 & -\til{b}}
  \end{align*}
  Passing to this subgroup's 1-connected cover gives a map:
  \begin{align*}
      h : \Sp(1) \times \Sp(1) \times \Sp(2) \to \Ff
  \end{align*}
  This map $h$ restricts to a double covering $h|_T$ of compatible
  maximal tori whose induced homomorphism on $\H^1$ corresponds to the
  inclusion of weight lattices:
  \begin{align*}
    \Z\langle a_1,a_2,a_3,a_4 \rangle \into
    \Z\langle -\tfrac12\til{a}, \tfrac12a_2, a_4-\tfrac12\til{b},
    a_4-\til{b} \rangle
    = \Z\langle a_1,\tfrac12a_2,a_3,a_4 \rangle
  \end{align*}

  Let $f : \HP^{n-4} \to \B\Ff$ denote the composition:
  \begin{align*}
    \HP^{n-4} \into \HP^\infty = \B\Sp(1) \xto{\Bi_1} \B(\Sp(1) \times
    \Sp(1) \times \Sp(3)) \xto{\B h} \B\Ff
  \end{align*}
  where $i_1 : \Sp(1) \into \Sp(1) \times \Sp(1) \times \Sp(3)$ is the
  inclusion of the first factor. The map $f$ classifies a Cayley plane
  bundle $\OP^2 \to E_n \to \HP^{n-4}$ fitting into a pullback
  diagram:
  \begin{align*}
    \xymatrix{
      E_n \ar[r]^-g \ar[d]_{\pi_f}
      & \B\Spin(9) \ar[d]^{\Bi_{\Spin(9),\Ff}} \\
      \HP^{n-4} \ar[r]^-f & \B\Ff
    }
  \end{align*}
  Use this diagram to compute:
  \begin{align*}
    s_n(p)[E_n]
    &= \int_{E_n} s_n(p)(TE_n)
    = \int_{E_n} s_n(p)(\pi_f^* T\HP^{n-4} \oplus g^*(\eta))
    = \int_{E_n} g^* s_n(p)(\eta)
    & \text{(Lemma~\ref{lemma:thom})} \\
    &= \int_{\HP^{n-4}} \pi_{f*} g^* s_n(p)(\eta) 
    = \int_{\HP^{n-4}} f^* \Bi_{\Spin(9),\Ff*} s_{n}(p)(\eta)
  \end{align*}

  Since the inclusion of the maximal torus $i_{S^1,\Sp(1)} : S^1 \into
  \Sp(1)$ induces an injection:
  \begin{align*}
    \H^*(\B\Sp(1),\Z)=\Z[\tfrac14\til{a}^2] \into 
    \Z[\tfrac12\til{a}]=\H^*(\B S^1,\Z)
  \end{align*}
  the pullback of $\Bi_{\Spin(9),\Ff*} s_{n}(p)(\eta) \in
  \H^{4n}(\B\Ff,\Z)$ to $\H^{4n}(\B\Sp(1),\Z)$ can be computed by pulling
  along the bottom of the diagram:
  \begin{align*}
    \xymatrix{
      \B\Sp(1) \ar[r]^-{\Bi_1}
      & \B(\Sp(1) \times \Sp(1) \times \Sp(2)) \ar[r]^-{\B h} &
      \B\Ff \\
      \B S^1 \ar[u] \ar[r]^{\B(i_1|_{S^1})} 
      & \BT^4 \ar[u] \ar[r]^{\B(h|_T)} & \BT^4 \ar[u]
    }
  \end{align*}

  Recall that Corollary~\ref{spin9f4pushforward} gives a formula for
  the image of $\Bi_{\Spin(9),\Ff*} s_{n}(p)(\eta)$ in
  $\H^{4n}(\BT^4,\Z)$. The composition $\B(i_1|_{S^1})^* \circ
  \B(h|_T)^*$ extracts the coefficient of $\tfrac12\til{a}$ with
  respect to the basis $\{\til{a},a_2,a_4,\til{b}\}$. Since:
  \begin{align*}
    (e_1,e_2,e_3,e_4)=\left(\tfrac12(\til{a}+\til{b}),
      \tfrac12(\til{a}-\til{b}), \tfrac12(a_2-2a_4-\til{b}),
      -\tfrac12(a_2+2a_4+\til{b})\right)
  \end{align*}
  it follows that the integral $\int_{\HP^{n-4}} f^*
  \Bi_{\Spin(9),\Ff*} s_{n}(p)(\eta)$ can be computed by taking the
  formula of Corollary~\ref{spin9f4pushforward} and substituting
  $(e_1,e_2,e_3,e_4) \mapsto (1,1,0,0)$. Some care is needed though
  since these substitutions make denominators vanish.  Substituting
  $(e_1,e_2,e_3,e_4)\mapsto(1,1+z,z,z)$ and applying l'H\^opital's
  rule six times with respect to $z$ gives:
%    & = -4^n + 4n^3 -6n^2 + 2n + 4
    \begin{align*}
      s_n(p)[E_n] = -\tfrac13(n-3)n(2n-1)(2n+1)
    \end{align*}
  which is strictly negative for $n\ge4$.
\end{proof}

\section{Proof of Theorem~\ref{THM:CAYLEY-PLANE-GENUS}}

As explained in the introduction, the elliptic genus and the Witten
genus:
\begin{align*}
  \phi_{ell} &: \MSO_* \tensor \Q \to \Q[\delta,\epsilon] \\
  \phi_W &: \MSO_* \tensor \Q \to \Q[\GG_2,\GG_4,\GG_6]
\end{align*}
are both known to be multiplicative for Cayley plane bundles. This
implies that the ideal $I=(E-\OP^2 \cdot B)$ is contained in the
kernel $K$ of the product:
\begin{align*}
  \phi_{ell} \times \phi_W : \MSO_* \tensor \Q \to
  \Q[\delta,\epsilon] \times \Q[\GG_2,\GG_4,\GG_6]
\end{align*}
To prove Theorem~\ref{thm:cayley-plane-genus}, we compute $K$ and then
show that $I=K$.

\medskip

The inclusion:
\begin{align*}
  \Q[\Kummer,\HP^2,\HP^3,\OP^2] \to \MSO_* \tensor \Q
\end{align*}
is an isomorphism in degrees 0 through 16. (The same is true if
$\MSO_*$ is replaced with $\MSpin_*$.) Here $\Kummer$ denotes the
Kummer surface, a 4-dimensional spin manifold with signature 16.

\begin{proposition}
  \label{prop:phiEllxphiWValues}
\begin{align*}
  \phi_{ell} \times \phi_W(\Kummer) &=(16 \delta,48 \GG_2) \\
  \phi_{ell} \times \phi_W(\HP^2) &= (\epsilon,2\GG_2^2-\tfrac56\GG_4) \\
  \phi_{ell} \times \phi_W (\HP^3)
  &=(0,-\tfrac49\GG_2^3+\tfrac19\GG_2\GG_4+\tfrac7{1080}\GG_6) \\
  \phi_{ell} \times \phi_W (\OP^2) &= (\epsilon^2,0)
\end{align*}  
\end{proposition}

\begin{proof}
  Use the following identities in $\MSO_* \tensor \Q$:
  \begin{align*}
    \Kummer &= 16 \CP^2 \\
    \HP^2&= 3(\CP^2)^2-2\CP^4 \\
    \HP^3&= \tfrac23 (\CP^2)^3 - \CP^2 \CP^4 + \tfrac13 \CP^6 \\
    \OP^2&= \tfrac{145}{3} (\CP^2)^4-92 (\CP^2)^2 \cdot \CP^4+36
    \CP^2 \cdot \CP^6+18 (\CP^4)^2-\tfrac{28}{3} \CP^8
  \end{align*}
%   Here Hirzebruch's formula $p(\HP^n) = (1+u)^{2n+2}(1+4u)^{-1}$ where
%   $u$ generates $\H^4(\HP^n,\Z)$ (Theorem~1.3 of \cite{hirzebruch-92})
%   is useful.
  To verify the first identity note that the signature restricts to an
  isomorphism $\MSO_4 \to \Z$. To verify the rest, compare Pontrjagin
  numbers. For the Pontrjagin numbers of $\CP^n$ use the formula
  $p(T\CP^n)=(1-g^2)^{n+1}$ where $g$ generates $\H^2(\CP^n,\Z)$. For
  the Pontrjagin numbers of $\HP^n$ use Hirzebruch's formula
  $p(T\HP^n) = (1+u)^{2n+2}(1+4u)^{-1}$ where $u$ generates
  $\H^4(\HP^n,\Z)$ (Theorem~1.3 of \cite{hirzebruch-92}). For the
  Pontrjagin numbers of $\OP^2$ see \cite[Theorem
  19.4]{borel-hirzebruch-1958}.

%   For the Pontrjagin numbers of $\CP^n$ use the formula:
%   \begin{align*}
%     s_I(p)(T\CP^n) &= \#\Big\{ \text{\small permutations of the list
%         $(\underbrace{i_1,\dots,i_r,0,\dots,0}_{\text{length
%             $n+1$}})$} \Big\} \cdot c_1(O(1))^{2(i_1+\cdots+i_r)}
%   \end{align*}
%   where $I=i_1,\dots,i_4$ is a partition of $n$.

  The values $\phi_{ell} \times \phi_W(\CP^{2n})$ can in turn be
  extracted from the characteristic power series of $\phi_{ell}$ and
  $\phi_W$ since any genus $\phi$ with characteristic power series $Q$
  satisfies:
  \begin{align*}
    g'(z) = \frac{d}{dz} \, \left( \frac{z}{Q(z)} \right)^{-1} = \;\;
    \sum_{n=1}^\infty \phi(\CP^{2n}) z^{2n}
  \end{align*}
  where the logarithm $g(z)=(z/Q(z))^{-1}$ is the formal power series
  satisfying $g(z/Q(z))= 1$.

  To extract $\phi_{ell} \times \phi_W (\CP^{2n})$ for $1 \le n \le 4$
  from the characteristic power series given in the introduction, use
  the identities:
  \begin{align*}
    \delta&=3\til{\GG}_2 &
    \til{\GG}_6 &= \tfrac{120}7(4\til{\GG}_2^3-\til{\GG}_2\til{\GG}_4)
    & \GG_8&=120\GG_4^2
    \\
    \epsilon&=\tfrac16(12\til{\GG}_2^2-5\til{\GG}_4) &
    \til{\GG}_8&=
    -\tfrac{20}3(144\til{\GG}_2^4-120\til{\GG}_2^2+7\til{\GG}_4^2)
    \tag*{\qedhere}
  \end{align*}
\end{proof}

\begin{corollary}
  \label{cor:kerphiEllxphiW}
  The kernel of the restriction:
  \begin{align*}
    \phi_{ell} \times \phi_W : \Q[\Kummer,\HP^2,\HP^3,\OP^2] \to
    \Q[\delta,\epsilon] \times \Q[\GG_2,\GG_4,\GG_6]
  \end{align*}
  is the ideal $(\OP^2) \cdot \big(\HP^3,\OP^2-(\HP^2)^2\big)$.
\end{corollary}

\begin{proof}
  Proposition~\ref{prop:phiEllxphiWValues} implies that the
  restriction of $\phi_{ell}$ splits as a tensor product:
  \begin{align*}
    \Q[\Kummer] \tensor \Q[\HP^3] \tensor \Q[\HP^2,\OP^2] \to
    \Q[\delta] \tensor \Q \tensor \Q[\epsilon]
  \end{align*}
  whose kernel is clearly the ideal $\big(\HP^3,\OP^2-(\HP^2)^2\big)$.
  It also implies that $\phi_W$ vanishes on $\OP^2$ and restricts to
  an isomorphism $\Q[\Kummer,\HP^2,\HP^3] \to
  \Q[\GG_2,\GG_4,\GG_6]$. The restriction of $\phi_{ell} \times
  \phi_W$ to $\Q[\Kummer,\HP^2,\HP^3,\OP^2]$ therefore has kernel:
  \begin{align*}
    (\OP^2) \cap \big(\HP^3,\OP^2-(\HP^2)^2\big) = (\OP^2) \cdot
    \big(\HP^3,\OP^2-(\HP^2)^2\big)
    \tag*{\qedhere}
  \end{align*}
\end{proof}

\bigskip

Corollary~\ref{cor:kerphiEllxphiW} implies that the kernel $K$ of
$\phi_{ell} \times \phi_W$ is generated by:
\begin{align*}
  \begin{cases}
    R_7 = \OP^2 \cdot \HP^3 \\
    R_8 = \OP^2 \cdot \big(\OP^2-(\HP^2)^2\big)
  \end{cases}
\end{align*}
together with the differences $E_n-\OP^2 \cdot \HP^{n-4}$ for $n \ge
4$ where $\OP^2 \to E_n \to \HP^{n-4}$ is the bundle constructed in
the proof of Theorem~\ref{thm:new-char-witten-genus}. The $\Q$-vector
space:
\begin{align*}
  V_n(K)=K_{4n}\big/\big(\sum_{0<i<n} K_{4i} \cdot MSO_{4n-4i}\big)
\end{align*}
therefore has dimension:
\begin{align*}
  \dim_\Q V_n(K)  =
  \begin{cases}
    1 & \text{for $n \ge 9$} \\
    2 & \text{for $7 \le n \le 8$} \\
    1 & \text{for $5 \le n \le 6$} \\
    0 & \text{for $1 \le n \le 4$}
  \end{cases}
\end{align*}

To prove that $I=K$ it suffices to show that the $\Q$-vector space:
\begin{align*}
  V_n(I)=I_{4n}\big/\big(\sum_{0<i<n} I_{4i} \cdot MSO_{4n-4i}\big)
\end{align*}
has the same dimension as $V_n(K)$ for each $n \ge 1$.
Theorem~\ref{thm:new-char-witten-genus} implies that $\dim_\Q V_n(I)
\ge 1$ for $n \ge 5$ so all that remains is to show that $\dim_\Q
V_n(I)=2$ for $n=7,8$. We do this by constructing two bundles $\OP^2
\to E'_7 \to \HP^2 \times \HP^1$, $\OP^2 \to E'_8 \to \HP^3 \times
\HP^1$ and showing that the images of $E_n-\OP^2 \cdot \HP^{n-4}$ and
$E'_n-\OP^2 \cdot \HP^{n-5} \cdot \HP^1$ are linearly independent in
$V_n(I)$ for $n=7,8$. To establish linear independence it suffices to
exhibit two Pontrjagin numbers $\alpha_n,\beta_n$ which vanish on
$\sum_{0<i<n} I_{4i} \cdot MSO_{4n-4i}$ and to check that the
determinant:
\begin{align*}
  \begin{vmatrix}
    \alpha_n\big(E_n-\OP^2 \cdot \HP^{n-4}\big)
    & \alpha_n\big(E'_n-\OP^2 \cdot \HP^{n-5} \cdot \HP^1\big) \\
    \beta_n\big(E_n-\OP^2 \cdot \HP^{n-4}\big)
    & \beta_n\big(E'_n-\OP^2 \cdot \HP^{n-5} \cdot \HP^1\big)
  \end{vmatrix}
\end{align*}
is nonzero.

\bigskip

Modify the construction of $\OP^2 \to E_n \to \HP^{n-4}$ in the proof
of Theorem~\ref{thm:new-char-witten-genus} as follows. Let $f :
\HP^{n-5} \times \HP^1 \to \B\Ff$ denote the composition:
\begin{align*}
  \HP^{n-5} \times \HP^1 \into \HP^\infty \times \HP^\infty =
  \B(\Sp(1) \times \Sp(1)) \xinto{\B(i_1 \times i_2)} \B(\Sp(1) \times
  \Sp(1) \times \Sp(3)) \xto{\B h} \B\Ff
\end{align*}
and let $\OP^2 \to E'_n \to \HP^{n-5} \times \HP^1$ denote the Cayley
plane bundle classified by $f$.

\begin{proposition}
  \label{prop:sIpEn}
  \begin{align*}
    s_7(p)[E'_7] &= -5824 & s_8(p)[E'_8] &= -15776 \\
    s_{4,3}(p)[E'_7] &= 9184 & s_{4,4}(p)[E'_8] &= 11024
  \end{align*}
\end{proposition}

\begin{proof}
  The calculation of $s_n(p)[E_n]$ in the proof of
  Proposition~\ref{prop:kill-higher-gens} can be adapted to compute
  $s_n(p)[E'_n]$. Instead of substituting $(e_1,e_2,e_3,e_4) \mapsto
  (1,1,0,0)$ into the formula of Corollary~\ref{spin9f4pushforward},
  substitute $(e_1,e_2,e_3,e_4) \mapsto (g_1,g_1,g_2,-g_2)$ where
  $g_1,g_2$ are indeterminants and then extract the coefficient of
  $g_1^{2n-10} g_2^2$. This coefficient can be extracted assuming $n
  \ge 6$ by differentiating twice with respect to $g_2$, dividing by
  2, applying l'H\^opital's rule 6 times with respect to $g_2$, and
  then substituting $(g_1,g_2)\mapsto(1,0)$. This leads, for $n \ge
  6$, to the formula:
  \begin{align*}
    s_n(p)[E'_n] = -\frac{1}{45} (n-4) n (2 n-1) (2 n+1) \left(2 n^2-7
      n+15\right)
  \end{align*}

  The numbers $s_{4,3}(p)[E'_7]$ and $s_{4,4}(p)[E'_8]$ can be
  computed similarly because, just as for $s_n(p)[E_n]$ and
  $s_n(p)[E'_n]$, the Pontrjagin class $p(\HP^{n-5} \times \HP^1)$
  does not affect the calculation. For instance Lemma~\ref{lemma:thom}
  implies that $s_3(p)(T(\HP^2 \times \HP^1))=0$ and hence that:
  \begin{align*}
    s_{4,3}(p)[E'_7] &= \int_{E'_7} s_{4,3}(p)(TE'_7) = \int_{E'_7}
    s_{4,3}(p)(\pi_f^* T(\HP^2 \times \HP^1) \oplus g^*(\eta)) =
    \int_{E'_7} g^* s_{4,3}(p)(\eta)
  \end{align*}
  The last integral can then be computed using the formula of
  Corollary~\ref{spin9f4pushforward} as above.
\end{proof}

% \begin{proposition}
%   \label{prop:sIpHPn}
%   \begin{align*}
%     s_3(p)[\HP^3] &= -56 & s_4(p)[\HP^4] &= -246
%   \end{align*}
% \end{proposition}
% \begin{proof}
%   This follows from Hirzebruch's formula (Theorem~1.3 of
%   \cite{hirzebruch-92}):
%   \begin{align*}
%     p(\HP^n) = (1+u)^{2n+2}(1+4u)^{-1}
%   \end{align*}
%   where $u$ generates $\H^4(\HP^n,\Z)$ and from the formulas:
%   \begin{align*}
%     s_3(p)&=p_1^3-3p_1p_2+3p_3 & s_4(p)&=p_1^4-4 p_2 p_1^2+4 p_3 p_1+2
%     p_2^2-4 p_4
%   \end{align*}
% \end{proof}

% I claim more generally that $s_n(p)[\HP^n] = -4^n+2n+2$ but I do not
% need this so I do not prove it.

\begin{proposition}
  \label{prop:s43pE7}
  \begin{align*}
    s_{4,3}(p)[E_7] &= 3164 + s_3[\HP^3]\;s_4[\OP^2] & s_{4,4}(p)[E_8]
    &= 2932 + s_4[\HP^4]\;s_4[\OP^2]
  \end{align*}
\end{proposition}
\begin{proof}
  The calculation is similar to that of $s_7(p)[E_7]$ and
  $s_8(p)[E_8]$ except that the Pontrjagin numbers of the base space
  $\HP^{n-4}$ begin to creep in. For instance:
  \begin{align*}
    s_{4,3}(p)[E_7] &= \int_{E_7} s_{4,3}(p)(TE_7)
    = \int_{E_7} s_{4,3}(p)(\pi_f^* T\HP^3 \oplus g^*\eta) \\
    &= \int_{\HP^3} s_3(p)(T\HP^3) \cdot f^* \Bi_{\Spin(9),\Ff*}
    s_4(p)(\eta) + f^* \Bi_{\Spin(9),\Ff*} s_{4,3}(p)(\eta) \\
    &= s_3[\HP^3]\;s_4[\OP^2] + 3164
    \tag*{\qedhere}
  \end{align*}
\end{proof}

In fact $s_4[\OP^2]=-84$ and $s_n(p)[\HP^n] = -4^n+2n+2$ but these
numbers drop out in the end.

% \begin{lemma}
%   \label{lemma:sIpCPn}
%   If $I=i_1,\dots,i_r$ then:
%   \begin{align*}
%     s_I(p)(T\CP^n) &= \#\Big\{ \text{permutations of the list
%         $(\underbrace{i_1,\dots,i_r,0,\dots,0}_{\text{length
%             $n+1$}})$} \Big\} \cdot c_1(O(1))^{2(i_1+\cdots+i_r)}
%   \end{align*}
%   where the generator $c_1(O(1))$ of $\H^2(\CP^n,\Z)$ satisfies
%   $c_1(O(1))^{n+1}=0$.
% \end{lemma}
% \begin{proof}
%   The formula $p(T\CP^n) = (1+c_1(O(1))^2)^{n+1}$ identifies
%   $p_i(T\CP^n)$ with the $i$th elementary symmetric polynomial in
%   $n+1$ indeterminants all set equal to $c_1(O(1))^2$. This formula
%   therefore also identifies $s_I(p)(T\CP^n)$ with the symmetric
%   polynomial $s_I$ in $n+1$ indeterminants all set equal to
%   $c_1(O(1))^2$.
% \end{proof}

\begin{proposition}
  \label{prop:alphan-betan}
  If $(\alpha_7,\beta_7)=(s_7(p),s_{4,3}(p))$ and
  $(\alpha_8,\beta_8)=(s_8(p),s_{4,4}(p))$ then $\alpha_n$ and
  $\beta_n$ vanish on $\sum_{0<i<n} I_{4i} \cdot MSO_{4n-4i}$ and the
  determinant:
  \begin{align*}
    \begin{vmatrix}
      \alpha_n\big(E_n-\OP^2 \cdot \HP^{n-4}\big)
      & \alpha_n\big(E'_n-\OP^2 \cdot \HP^{n-5} \cdot \HP^1\big) \\
      \beta_n\big(E_n-\OP^2 \cdot \HP^{n-4}\big) &
      \beta_n\big(E'_n-\OP^2 \cdot \HP^{n-5} \cdot \HP^1\big)
    \end{vmatrix}
  \end{align*}
  is nonzero for $n=7,8$.
\end{proposition}

\begin{proof}
  If $n=7$ then the $\Q$-vector space $\sum_{0<i<n} I_{4i} \cdot
  MSO_{4n-4i}$ is spanned by bordism classes of the form:
  \begin{center}
    $(E_5-\HP^1 \cdot \OP^2) \cdot M_2$ \\
    $(E_6-\HP^2 \cdot \OP^2) \cdot M_1$
  \end{center}
  where $M_i$ denotes a closed oriented manifold of real dimension
  $4i$. Lemma~\ref{lemma:thom} implies that $\alpha_7=s_7(p)$ and
  $\beta_7=s_{4,3}(p)$ vanish on all such bordism classes. By
  Propositions~\ref{prop:sIpEn} \& \ref{prop:s43pE7}:
  \begin{align*}
    \begin{vmatrix}
      \alpha_7\big(E_7-\OP^2 \cdot \HP^3\big)
      & \alpha_7\big(E'_7-\OP^2 \cdot \HP^2 \cdot \HP^1\big) \\
      \beta_7\big(E_7-\OP^2 \cdot \HP^3\big) &
      \beta_7\big(E'_7-\OP^2 \cdot \HP^2\cdot\HP^1\big)
    \end{vmatrix}
    =
    \begin{vmatrix}
      -1820 & -5824 \\
      3164 & 9184
    \end{vmatrix}
    \ne 0
  \end{align*}

  \medskip

  If $n=8$ then the case $n=7$ proved above implies that $\sum_{0<i<n}
  I_{4i} \cdot MSO_{4n-4i}$ is spanned by bordism classes of the form:
  \begin{center}
    $(E_5-\HP^1 \cdot \OP^2) \cdot M_3$ \\
    $(E_6-\HP^2 \cdot \OP^2) \cdot M_2$ \\
    $(E_7-\HP^3 \cdot \OP^2) \cdot M_1$
    \quad $(E'_7-\HP^3 \cdot \OP^2) \cdot M_1$
  \end{center}
  Lemma~\ref{lemma:thom} implies that $\alpha_8=s_8(p)$ and
  $\beta_8=s_{4,4}(p)$ vanish on all of them. By
  Propositions~\ref{prop:sIpEn} \& \ref{prop:s43pE7}:
  \begin{align*}
    \begin{vmatrix}
      \alpha_8\big(E_8-\OP^2 \cdot \HP^4\big)
      & \alpha_8\big(E'_8-\OP^2 \cdot \HP^3 \cdot \HP^1\big) \\
      \beta_8\big(E_8-\OP^2 \cdot \HP^4\big) &
      \beta_8\big(E'_8-\OP^2 \cdot \HP^3 \cdot \HP^1\big)
    \end{vmatrix}
    =
    \begin{vmatrix}
      -3400 & -15776 \\
      2932 & 11024
    \end{vmatrix}
    \ne 0
    \tag*{\qedhere}
  \end{align*}
\end{proof}

\bigskip

The final step is to show that the point of intersection
$\OP^2=\HP^3=\HP^2=0$ corresponds to the $\hat{A}$~genus. This follows
from Proposition~\ref{prop:phiEllxphiWValues} together with the fact
that the $\hat{A}$~genus is the point
$[\delta,\epsilon]=[-\tfrac18,0]$ of $\Proj\;\Q[\delta,\epsilon]$.
Alternatively it follows since $\HP^2,\HP^3,\OP^2$ are homogeneous
spaces and hence admit metrics of positive scalar curvature and
therefore have $\hat{A}=0$ by Lichnerowicz's theorem
\cite{lichnerowicz-1963}.

\section{Proof of Theorem~\ref{THM:MSPIN-CASE}}

Since $\OP^2$ and $\HP^n$ are spin manifolds so are the total spaces
$E_n$ and $E'_n$. Therefore, since the forgetful map $\MSpin_* \to
\MSO_*$ is an isomorphism tensor $\Q$, the proofs of
Theorems~\ref{thm:new-char-witten-genus} \&
\ref{thm:cayley-plane-genus} already prove
Theorem~\ref{thm:mspin-case}.

\section*{Acknowledgments}

Thanks to my PhD supervisor Burt Totaro for his guidance. Thanks to
Baptiste Calm\`es and Artie Prendergast-Smith for helpful
discussions. Thanks to the National Science Foundation, the Cambridge
Overseas Trusts, the Cambridge Philosophical Society and the Cambridge
Lundgren Fund for their generous financial support.

%\bibliographystyle{alpha}
%\bibliography{refs}

\end{document}